\documentclass[a4, 12pt]{amsart}
\usepackage{amssymb}
\usepackage{amstext}
\usepackage{amsmath}
\usepackage{amscd}
\usepackage{latexsym}
\usepackage{amsfonts}

\theoremstyle{plain}
\newtheorem{thm}{Theorem}[section]
\newtheorem*{thm*}{Theorem}
\newtheorem*{cor*}{Corollary}

\newtheorem{prop}[thm]{Proposition}
\newtheorem{lem}[thm]{Lemma}
\newtheorem{cor}[thm]{Corollary}

\newtheorem*{claim*}{Claim}

\theoremstyle{definition}
\newtheorem{defn}[thm]{Definition}
\newtheorem{ex}[thm]{Example}
\newtheorem{rem}[thm]{Remark}

\theoremstyle{remark}

\numberwithin{equation}{thm}

\def\Ker{\mathrm{Ker}}
\def\Im{\mathrm{Im}}

\def\e{\mathrm{e}}
\def\m{\mathfrak m}
\def\n{\mathfrak n}

\def\p{\mathfrak p}
\def\q{\mathfrak q}

\def\Z{\Bbb Z}

\newcommand{\rmG}{\mathrm{G}}

\newcommand{\rmR}{\mathrm{R}}
\newcommand{\rmS}{\mathrm{S}}

\def\depth{\mathrm{depth}}

\def\Ann{\mathrm{Ann}}
\def\Ass{\mathrm{Ass}}

\newcommand{\mapright}[1]{%
\smash{\mathop{%
\hbox to 1cm{\rightarrowfill}}\limits^{#1}}}

\newcommand{\mapleft}[1]{%
\smash{\mathop{%
\hbox to 1cm{\leftarrowfill}}\limits_{#1}}}

\newcommand{\mapdown}[1]{\Big\downarrow
\llap{$\vcenter{\hbox{$\scriptstyle#1\,$}}$ }}

\begin{document}

\setlength{\baselineskip}{12pt}
\title{The first Hilbert coefficient of stretched ideals}
\pagestyle{plain}
\author{Kazuho Ozeki}
\address{Department of Mathematical Sciences, Faculty of Science, Yamaguchi University, 1677-1 Yoshida, Yamaguchi 753-8512, Japan}
\email{ozeki@yamaguchi-u.ac.jp}
\thanks{
\endgraf
\endgraf
This is not in final form. The final/detailed version will be submitted to elsewhere for publication.}
\maketitle
\begin{abstract}
In this paper we explore the almost Cohen-Macaulayness of the associated graded ring of stretched $\m$-primary ideals with small first Hilbert coefficient in a Cohen-Macaulay local ring $(A,\m)$.
In particular, we explore the structure of stretched $\m$-primary ideals satisfying the equality $\e_1(I)=\e_0(I)-\ell_A(A/I)+4$ where $\e_0(I)$ and $\e_1(I)$ denote the multiplicity and the first Hilbert coefficient respectively.
\end{abstract}


\section{Introduction}

Throughout this paper, let $A$ be a Cohen-Macaulay local ring with maximal ideal $\m$ and $d=\dim A>0.$ 
For simplicity, we may assume the residue class field $A/\m$ is infinite.
Let $I$ be an $\m$-primary ideal in $A$ and $Q=(a_1,a_2,\ldots,a_d) \subseteq I$ be a parameter ideal in $A$ which forms a reduction of $I$.
Let 
$$ R=\rmR(I) := A[I t] \ \ \ \subseteq A[t] ~~\operatorname{and}~~ \ R'= \rmR'(I):= A[It,t^{-1}] \ \ \subseteq A[t,t^{-1}]$$
denote, respectively, the Rees algebra and the extended Rees algebra of $I$. 
Let
$$ G=\rmG(I):= R'/t^{-1}R' \cong \bigoplus_{n \geq 0}I^n/I^{n+1}$$
denote the associated graded ring of $I$.
Let $\ell_A(N)$ denote, for an $A$-module $N$, the length of $N$.

\vspace{2mm}

It is well known that the multiplicity $\e_0(\m)$ of a Cohen-Macaulay local ring $(A, \m)$ is written as $\e_0(\m) = \mu(\m)-d+k_{\m}$ for some integer $k_{\m} \geq 1$, where $\mu(\m)$ denotes the embedding dimension of $A$ $($c.f. \cite{A}$)$.
Sally investigated Cohen-Macaulay local rings with $k_{\m}=1$ $($resp. $k_{\m}=2)$ which have been called {\it rings of minimal multiplicity} $($resp. {\it rings of almost minimal multiplicity}$)$.
In \cite{S1}, Sally showed that the associated graded ring ${\rm G}(\m)$ of the maximal ideal $\m$ is always Cohen-Macaulay if $k_{\m}=1$.
In \cite{S3}, she proved that the associated graded ring ${\rm G}(\m)$ is Cohen-Macaulay if $k_{\m}=2$ and $\tau(A) < \mu(\m)-d$ where $\tau(A)$ denotes the Cohen-Macaulay type of $A$, while she found examples of Cohen-Macaulay local rings with $k_{\m}=2$ whose associated graded ring ${\rm G}(\m)$ is not Cohen-Macaulay.
Based on these considerations, Sally expected that the associated graded ring ${\rm G}(\m)$ is almost Cohen-Macaulay, i.e. $\depth {\rm G}(\m) \geq d-1$, if $k_{\m}=2$.
This {\it Sally's conjecture} was solved by Rossi-Valla \cite{RV1}, and independently Wang \cite{W}.
Thereafter, Rossi and Valla \cite{RV2} examined the almost Cohen-Macaulay property of the associated graded rings with $k_{\m}=3$, however the problem is still open.

\vspace{2mm}

The notion of {\it stretched} Cohen-Macaulay local rings was introduced by J. Sally to extend the rings of almost minimal multiplicity.
We say that the ring $A$ is {\it stretched} if $\ell_A(Q+\m^2/Q+\m^3)=1$ holds true, i.e. the ideal $(\m/Q)^2$ is principal, for some parameter ideal $Q$ in $A$ which forms a reduction of $\m$ $($\cite{S2}$)$.
We note here that this condition depends on the choice of a reduction $Q$ $($see \cite[Example 2.3]{RV3}$)$.
Sally \cite{S2} showed that the equality $\m^{k_{\m}+1}=Q\m^{k_{\m}}$ holds true if and only if the associated graded ring ${\rmG}(\m)$ of $\m$ is Cohen-Macaulay if $A$ is stretched.

These arguments work for $\m$-primary ideals in Cohen-Macaulay local rings.
One knows that, for a given $\m$-primary ideal $I$, the equality $\e_0(I)=\ell_A(I/I^2)-(d-1)\ell_A(A/I)+k_I-1$ holds true for some integer $k_I \geq 1$ where $\e_0(I)$ denotes the multiplicity of $I$.
We see that the associated graded ring $G$ of $I$ is Cohen-Macaulay if $k_I=1$ $($c.f. \cite{VV}$)$, and $G$ is almost Cohen-Macaulay if $k_I=2$ $($c.f. \cite{R, RV1, W}$)$.

In 2001, Rossi and Valla \cite{RV3} gave the notion of stretched $\m$-primary ideals.
We say that the $\m$-primary ideal $I$ is stretched if the following two conditions
\begin{itemize}
\item[$(1)$] $Q \cap I^2=QI$ and
\item[$(2)$] $\ell_A(Q+I^2/Q+I^3)=1$
\end{itemize}
hold true for some parameter ideal $Q$ in $A$ which forms a reduction of $I$.
We notice that the first condition is naturally satisfied if $I=\m$ so that this extends the classical definition of stretched local rings given in \cite{S2}.

Throughout this paper, a stretched $\m$-primary ideal $I$ will come always equipped with a parameter ideal $Q$ in $A$ which forms a reduction of $I$ such that $Q \cap I^2=QI$ and $\ell_A(I^2+Q/I^3+Q)=1$.
Rossi and Valla \cite{RV3} showed that the equality $I^{{k_I}+1}=QI^{k_I}$ holds true if and only if the associated graded ring $G$ of $I$ is Cohen-Macaulay if $I$ is stretched.
They also showed that the associated graded ring $G$ of $I$ is almost Cohen-Macaulay if $I$ is stretched and $I^{k_I+1} \subseteq QI^{k_I-1}$ holds true $($\cite[Proposition 3.1]{RV3}$)$. 
Moreover, Mantelo and Xie \cite{MX} introduced the notion of $j$-stretched ideals and generalized the results of Sally and Rossi-Valla.
An interesting result on the depth of fiber cones of stretched $\m$-primary ideals is presented in \cite{JN}.
Recently, stretched $\m$-primary ideals with small reduction number were explored in \cite{O3}.

\vspace{2mm}

The first Hilbert coefficient of $\m$-primary ideal in a local ring is an important numerical invariant associated to an ideal.
In this paper we explore the first Hilbert coefficient of stretched $\m$-primary ideals.

As is well known, for a given $\m$-primary ideal $I$, there exist integers $\{\e_k(I)\}_{0 \leq k \leq d}$ such that the equality
$$\ell_A(A/I^{n+1})={\e}_0(I)\binom{n+d}{d}-{\e_1}(I)\binom{n+d-1}{d-1}+\cdots+(-1)^d{\e}_d(I)$$
holds true for all integers $n \gg 0$.
For each $0 \leq k \leq d$, $\e_k(I)$ is called the $k$-th {\it Hilbert coefficient} of $I$, and 
especially, the leading coefficient $\e_0(I)$ is the multiplicity of $I$.
It is also well known that the inequality
$$ \e_1(I) \geq \e_0(I)-\ell_A(A/I) $$
holds true for any $\m$-primary ideal $I$ $($\cite{N}$)$. 
The equality $\e_1(I)=\e_0(I)-\ell_A(A/I)$ holds ture if and only if $I^2=QI$ is satisfied, and when this the case the associated graded ring $G$ of $I$ is Cohen-Macaulay by \cite{Hun,O}.
The complete structure theorem of the associated graded ring $G$ of $I$ satisfying the equality $\e_1(I)=\e_0(I)-\ell_A(A/I)+1$ was given in \cite{GNO}.
We see that the associated graded ring $G$ of $I$ is almost Cohen-Macaulay, if $\e_1(I)=\e_0(I)-\ell_A(A/I)+2$ and $Q \cap I^2=QI$ hold true $($see \cite[Theorem 4.6]{RV4}$)$.
In this paper we will show that the associated graded ring $G$ of $I$ is almost Cohen-Macaulay if $I$ is stretched and the equality $\e_1(I)=\e_0(I)-\ell_A(A/I)+3$ holds true $($see Proposition \ref{rank3}$)$.

The main purpose of this paper is to explore the structure of the stretched $\m$-primary ideals with the equality $\e_1(I)=\e_0(I)-\ell_A(A/I)+4$ which is the content of Section 5.
We present in this case the Hilbert coefficients and series, and in particular, we prove that $\depth G = d-1$ $($Theorem \ref{rank4}$)$.

The almost Cohen-Macaulayness of the associated graded ring of stretched $\m$-primary ideals with small first Hilbert coefficient will be given as follows.

\begin{cor}[Corollary \ref{cor4}]
Suppose that $I$ is stretched and assume that $\e_1(I) \leq \e_0(I)-\ell_A(A/I)+4$, then $G$ is almost Cohen-Macaulay $($i.e. $\depth G \geq d-1)$.
\end{cor}

Let us briefly explain how this paper is organized.
In Section 2, we will summarize some auxiliary results on the stretched $\m$-primary ideals.
In Section 3, we will introduce some techniques of the Sally module $S={\rmS}_Q(I)=I{\rmR}(I)/I{\rm R}(Q)$ for computing the Hilbert coefficients of $I$, some of them are stated in a general setting. 
In particular, the graded module $S^{(2)}$ will be the key ingredient for proving the main theorem.
In Section 4, we shall discuss the first Hilbert coefficient of stretched $\m$-primary ideals. 
In Section 5, we shall explore the stretched $\m$-primary ideals with small first Hilbert coefficient, and prove our main theorem.
In Section 6, we will introduce some examples of stretched local rings whose maximal ideal satisfying conditions of the main theorem.


\section{Preliminary Steps}

The purpose of this section is to summarize some results on the structure of the stretched $\m$-primary ideals, which we need throughout this paper.

Recall that, for a given $\m$-primary ideal $I$, the equality $\e_0(I)=\ell_A(I/I^2)-(d-1)\ell_A(A/I)+k_I-1$ holds true for some integer $k_I \geq 1$.
We set $r_{I}=r_Q(I):=\min\{n \geq 0 \ | \ I^{n+1}=QI^n\}$ denotes the reduction number of $I$ with respect to $Q$.

Let us begin with the following remark.

\begin{rem}
The following assertions hold true.
\begin{itemize}
\item[$(1)$] We notice that $\ell_A(I^2/QI)=\e_0(I)+(d-1)\ell_A(A/I)-\ell_A(I/I^2)=k_I-1$ holds true (see for instance \cite{RV4}), so that  $\ell_A(I^2/QI)$ and $k_I$ do not depend on a choice of minimal reduction $Q$ of $I$.
\item[$(2)$] Suppose that $I$ is stretched. 
Then the equality $k_I=n_I$ holds true where $n_I=\min\{n \geq 0 \ | \ I^{n+1} \subseteq Q \}$ denotes the index of nilpotency of $I$ with respect to $Q$ $($see Lemma \ref{length}$)$. 
Therefore it is easy to see that the inequality $k_I \leq r_I$ holds true for any stretched $\m$-primary ideals $I$.
\end{itemize}
\end{rem}

We have the following lemma which was given by Rossi and Valla.

\begin{lem}{$($\cite[Lemma 2.4]{RV1}$)$}\label{powers}
Suppose that $I$ is stretched.
Then we have the following.
\begin{itemize}
\item[$(1)$] There exist $x, y \in I \backslash Q$ such that $I^{n+1}=QI^n+(x^ny)$ holds true for all $n \geq 1$.
\item[$(2)$] The map $$ I^{n+1}/QI^n \overset{\widehat{x}}{\to}I^{n+2}/QI^{n+1} $$ is surjective for all $n \geq 1$. Therefore $\ell_A(I^{n+1}/QI^n) \geq \ell_A(I^{n+2}/QI^{n+1})$ for all $n \geq 1$.
\item[$(3)$] $\m x^{n}y \subseteq QI^{n}+I^{n+2}$ and hence $\ell_A(I^{n+1}/QI^n+I^{n+2}) \leq 1$ for all $n \geq 1$.
\end{itemize}
\end{lem}

We set $$\Lambda:=\Lambda_I=\Lambda_Q(I)=\{n \geq 1 \ | \ QI^{n-1}\cap I^{n+1}/QI^n \neq (0) \}$$
and $|\Lambda|$ denotes the cardinality of the set $\Lambda$.

The following lemma seems well known, but let us give a proof of it for the sake of completeness.

\begin{lem}\label{length}
Suppose that $I$ is stretched.
Then we have the following.
\begin{itemize}
\item[$(1)$] $\ell_A(I^{n+1}+Q/Q)=k_I-n$ for $1 \leq n \leq k_I$. Hence, especially, $\ell_A(I^2/QI)=\ell_A(I^2+Q/Q)=k_I-1$.
\item[$(2)$] For $n \geq 2$, 
\[\ell_A(I^{n+1}/QI^n)=\left\{
\begin{array}{ll}
\ell_A(I^n/QI^{n-1})-1 & \quad \mbox{if $n \notin \Lambda$,} \\
\ell_A(I^n/QI^{n-1}) & \quad \mbox{if $n \in \Lambda$.}
\end{array}
\right.\]
\end{itemize}
\end{lem}

\begin{proof}
$(1)$: Since $I^{k+1}+Q/I^{k+2}+Q \cong I^{k+1}/Q \cap I^{k+1}+I^{k+2}$ and 
$$0 \neq \ell_A(I^{k+1}+Q/I^{k+2}+Q)=\ell_A( I^{k+1}/Q \cap I^{k+1}+I^{k+2}) \leq \ell_A(I^{k+1}/QI^k+I^{k+2}) \leq 1$$
by Lemma \ref{powers} $(3)$, we have $\ell_A(I^{k+1}+Q/I^{k+2}+Q)=1$ for all $n \leq k \leq k_I-1$.
Hence, we have $\ell_A(I^{n+1}+Q/Q)= \sum_{k=n}^{k_I-1}\ell_A(I^{k+1}+Q/I^{k+2}+Q)=k_I-n$
as required.

\noindent
$(2)$: It is enough to show the case where $2 \leq n \leq r_I-1$.
For all $2 \leq n \leq r_I-1$, we have $\ell_A(I^{n}/QI^{n-1}+I^{n+1})=1$ by Lemma \ref{powers} $(3)$ and $QI^{n-1}+I^{n+1}/QI^{n-1} \cong I^{n+1}/QI^{n-1} \cap I^{n+1}$ so that the equalities 
\begin{eqnarray*}
\ell_A(I^{n+1}/QI^{n}) &=& \ell_A(I^{n+1}/QI^{n-1} \cap I^{n+1})+\ell_A(QI^{n-1} \cap I^{n+1}/QI^{n})\\
&=& \ell_A(QI^{n-1}+I^{n+1}/QI^{n-1})+\ell_A(QI^{n-1} \cap I^{n+1}/QI^{n})\\
&=& \{\ell_A(I^{n}/QI^{n-1})-\ell_A(I^{n}/QI^{n-1}+I^{n+1})\}+\ell_A(QI^{n-1} \cap I^{n+1}/QI^{n})\\
&=& \ell_A(I^{n}/QI^{n-1})-1+\ell_A(QI^{n-1} \cap I^{n+1}/QI^{n})
\end{eqnarray*}
hold true.
Then since $\ell_A(QI^{n-1} \cap I^{n+1}/QI^{n}) = 1+\ell_A(I^{n+1}/QI^{n})-\ell_A(I^{n}/QI^{n-1}) \leq 1$ by Lemma \ref{powers} $(2)$, we get the required condition.
\end{proof}

We furthermore have the following proposition.

\begin{prop}\label{sum}{$($\cite[Proposition 2.3]{O3}$)$}
Suppose that $I$ is stretched.
Then the following assertions hold true.
\begin{itemize}
\item[$(1)$] $|\Lambda|=r_I-k_I$ and
\item[$(2)$] $\displaystyle  \sum_{n \geq 1}\ell_A(I^{n+1}/QI^n)=\binom{r_I}{2}-\sum_{s \in \Lambda}s+|\Lambda| .$
\end{itemize}
\end{prop}

Thanks to the theorem of Huckaba \cite[Theorem 3.1]{Huk} $($also Huckaba and Marley \cite[Theorem 4.7]{HM}$)$ and by Proposition \ref{sum}, we get the following upper bound of the first Hilbert coefficient $\e_1(I)$ of stretched $\m$-primary ideals $I$ as a corollary.

\begin{cor}[c.f. \cite{Huk, HM}]\label{HM}
Suppose that $I$ is stretched.
Then the equality $$\e_1(I) \leq \e_0(I)-\ell_A(A/I)+\binom{r_I}{2}-\sum_{s \in \Lambda}s+|\Lambda|$$
holds true, and the following two conditions are equivalent;
\begin{itemize}
\item[$(1)$] $\e_1(I) = \e_0(I)-\ell_A(A/I)+\binom{r_I}{2}-\sum_{s \in \Lambda}s+|\Lambda|$ and
\item[$(2)$] $G$ is almost Cohen-Macaulay $($i.e. $\depth G \geq d-1)$.
\end{itemize}
\end{cor}

\section{The structure of Sally modules}

The purpose of this section is to summarize some results and techniques on the Sally modules which we need throughout this paper.
Remark that in this section $\m$-primary ideals $I$ are not necessarily stretched.  

Let $T={\rmR}(Q)=A[Qt] \subseteq A[t]$ denotes the Rees algebra of $Q$.
Following Vasconcelos \cite{V}, we consider 
$$S=\rmS_Q(I)=IR/IT \cong \bigoplus_{n \geq 1}I^{n+1}/Q^nI$$ 
the Sally module of $I$ with respect to $Q$.

We give one remark about Sally modules.
See \cite{GNO, V} for further information.

\begin{rem}[\cite{GNO,V}]
We notice that $S$ is a finitely generated graded $T$-module and $\m^{n}S=(0)$ for all $n \gg 0$.
We have $\Ass_TS \subseteq \{\m T\}$ so that $\dim_TS=d$ if $S \neq (0)$.
\end{rem}

From now on, let us introduce some techniques, being inspired by \cite{C, CPV}, which play a crucial role throughout this paper.
See \cite[Section 3]{O2} $($also \cite[Section 2]{O1} for the case where $I=\m)$ for the detailed proofs.

We denote by $E(m)$, for a graded module $E$ and each $m \in \Z$, the graded module whose grading is given by $[E(m)]_n=E_{m+n}$ for all $n \in \Z$.

We have an exact sequence
$$ 0 \to K^{(-1)} \to F \overset{\varphi_{-1}}{\to} G \to R/I R+T \to 0 \ \ \ \ (\dagger_{-1})$$
of graded $T$-modules induced by tensoring the canonical exact sequence
$$ 0 \to T \overset{i}{\hookrightarrow} R \to R/T \to 0$$
of graded $T$-modules with $F$, where $\varphi_{-1}=F \otimes i$, $K^{(-1)}=\Ker \varphi_{-1}$, and $F=T/IT \cong (A/I)[X_1,X_2,\cdots,X_d]$ is a polynomial ring with $d$ indeterminates over the residue class ring $A/I$.
We set $\mu_A(N)=\ell_A(N/\m N)$ denotes the minimal number of generators of an $A$-module $N$.

\begin{lem}$($\cite[Lemma 3.1]{O2}$)$\label{seq}
Assume that $I \supsetneq Q$ and put $\mu=\mu_A(I/Q)$. 
Then there exists an exact sequence
$$ T(-1)^{\mu} \overset{\phi}{\to} R/T \to S(-1) \to 0 $$
as graded $T$-modules.
\end{lem}

Tensoring the exact sequence of Lemma \ref{seq} with $F$, we get an exact sequence
$$ F(-1)^{\mu} \overset{\overline{\phi}}{\to} R/I R+T \to (S/I S)(-1) \to 0$$
of graded $F$-modules, where $\overline{\phi}=A/I \otimes \phi$.

We furthermore get the following commutative diagram
\[\begin{array}{ccccc}
 (F_{0}^{\mu} \otimes F)(-1) & \to & ([R/IR+T]_1 \otimes F)(-1) & \to & 0 \\
\mapdown{\simeq} & & \mapdown{\varphi_0} & & \\
 F(-1)^{\mu} & \to & \Im \overline{\phi} & \to & 0 
\end{array}\]
of graded $F$-modules by tensoring the sequence $$F_0^{\mu} \to [R/IR+T]_1 \to 0$$ with $F$.
Then, we get the exact sequence
$$ 0 \to K^{(0)}(-1) \to ([R/IR+T]_1 \otimes F)(-1) \overset{\varphi_0}{\to} R/IR+T \to S/IS(-1) \to 0 \ \ \ (\dagger_0)$$
of graded $F$-modules where $K^{(0)}=\Ker \varphi_0$.

Notice that $\Ass_TK^{(m)} \subseteq \{\m T\}$ for all $m = -1, 0$, because $F \cong (A/I)[X_1,X_2,\cdots,X_d]$ is a polynomial ring over the residue ring $A/I$ and $[R/IR+T]_1 \otimes F$ is a maximal Cohen-Macaulay module over $F$.

We then have the following proposition by the exact sequences $(\dagger_{-1})$ and $(\dagger_0)$.

\begin{prop}$($\cite[Lemma 3.3]{O2}$)$\label{EV}
We have
\begin{eqnarray*}
\ell_A(I^n/I^{n+1})&=&\ell_A(A/[I^2+Q])\binom{n+d-1}{d-1}-\ell_A(I/[I^2+Q])\binom{n+d-2}{d-2}\\
&+& \ell_A([S/IS]_{n-1})-\ell_A(K^{(-1)}_n)-\ell_A(K^{(0)}_{n-1})
\end{eqnarray*}
for all $n \geq 0$.
\end{prop}

\vskip 2mm
We also need the notion of {\it{filtration of the Sally module}} which was introduced by M. Vaz Pinto \cite{VP} as follows.

\begin{defn}(\cite{VP})
We set, for each $m \geq 1$,
$$ S^{(m)}=I^{m}t^{m-1}R/I^{m}t^{m-1}T (\cong I^{m}R/I^{m}T(-m+1)).$$ 
\end{defn}

We notice that $S^{(1)}=S$, and $S^{(m)}$ are finitely generated graded $T$-modules for all $m \geq 1$, since $R$ is a module-finite extension of the graded ring $T$.

The following lemma follows by the definition of the graded module $S^{(m)}$.

\begin{lem}\label{fact1}
Let $m \geq 1$ be an integer.
Then the following assertions hold true.
\begin{itemize}
\item[$(1)$] $\m^{n} S^{(m)} = (0)$ for integers $n \gg 0$; hence ${\dim}_TS^{(m)} \leq d$.
\item[$(2)$] The homogeneous components $\{ S^{(m)}_n \}_{n \in \Z}$ of the graded $T$-module $S^{(m)}$ are given by
\[ S^{(m)}_n \cong  \left\{
\begin{array}{rl}
(0) & \quad \mbox{if $n \leq m-1 $,} \\
I^{n+1}/Q^{n-m+1}I^{m} & \quad \mbox{if $n \geq m$.}
\end{array}
\right.\]
\end{itemize}
\end{lem}

In this paper, we need some structure theorem of $S^{(2)}$.
Let us introduce one proposition about it.

\begin{prop}$($\cite[Proposition 2.2 and 2.9]{OR}$)$\label{S2}
Suppose that $Q \cap I^2=QI$.
Then the following assertions hold true where $\p=\m T$.
\begin{itemize}
\item[$(1)$] $\Ass_T S^{(2)} \subseteq \{\p\}$ so that $\dim_TS^{(2)}=d$ if $S^{(2)} \neq (0)$.
\item[$(2)$] $\ell_{T_{\p}}(S^{(2)}_{\p})=\e_1(I)-\e_0(I)+\ell_A(A/I)-\ell_A(I^2/QI)$.
\end{itemize}
\end{prop}

Let $L^{(m)}= T S_m^{(m)} $ be a graded $T$-submodule of $S^{(m)}$ generated by $S_m^{(m)}$ and
\begin{eqnarray*}
D^{(m)} &=& (I^{m+1}/QI^m) \otimes (A/\Ann_A(I^{m+1}/QI^m))[X_1,X_2,\cdots,X_d]\\
&\cong& (I^{m+1}/QI^m)[X_1,X_2,\cdots,X_d]
\end{eqnarray*}
for $m \geq 1$ $($c.f. \cite[Section 2]{VP}$)$.

We then have the following lemma.

\begin{lem}$($\cite[Section 2]{VP}$)$\label{VP}
The following assertions hold true for $m \geq 1$.
\begin{itemize}
\item[$(1)$] $S^{(m)}/L^{(m)} \cong S^{(m+1)}$ so that the sequence
$$ 0 \to L^{(m)} \to S^{(m)} \to S^{(m+1)} \to 0 $$
 is exact as graded $T$-modules.
\item[$(2)$] There is a surjective homomorphism $\theta_m: D^{(m)}(-m) \to L^{(m)}$ of graded $T$-modules.
\end{itemize}
\end{lem}

For each $m \geq 1$, tensoring the exact sequence
$$ 0 \to L^{(m)} \to S^{(m)} \to S^{(m+1)} \to 0$$
and the surjective homomorphism $\theta_m:D^{(m)}(-m) \to L^{(m)}$ of graded $T$-modules with $F$, we get the exact sequence
$$ 0 \to K^{(m)}(-m) \to \overline{D^{(m)}}(-m) \overset{\varphi_{m}}{\to} \overline{S^{(m)}} \to \overline{S^{(m+1)}} \to 0 \ \ \ \ (\dagger_m) $$
of graded $F$-modules where $\overline{D^{(m)}}=D^{(m)}/ID^{(m)}$, $\overline{S^{(m)}}=S^{(m)}/IS^{(m)}$, and $K^{(m)}=\Ker \varphi_{m}$.

Notice here that, for all $m \geq 1$, we have $\Ass_T K^{(m)} \subseteq \{\m T\}$ because $\overline{D^{(m)}} \cong (I^{m+1}/QI^m+I^{m+2})[X_1,X_2,\cdots,X_d]$ is a maximal Cohen-Macaulay module over $F$.

We then have the following proposition.

\begin{prop}\label{function}{$($\cite[Proposition 3.7]{O3}$)$}
The following assertions hold true:
\begin{itemize}
\item[$(1)$] We have 
\begin{eqnarray*}
\ell_A(I^n/I^{n+1}) &=& \{\ell_A(A/I^2+Q)+\sum_{m=1}^{r_I-1}\ell_A(I^{m+1}/QI^m+I^{m+2})\}\binom{n+d-1}{d-1}\\
&+& \sum_{k=1}^{r_I}(-1)^k\left\{\sum_{m=k-1}^{r_I-1}\binom{m+1}{k}\ell_A(I^{m+1}/QI^m+I^{m+2})\right\}\binom{n+d-k-1}{d-k-1}\\
&-& \sum_{m=-1}^{r_I-1}\ell_A(K^{(m)}_{n-m-1})
\end{eqnarray*}
for all $n \geq \max\{0, r_I-d+1\}$.
\item[$(2)$] $\displaystyle \e_0(I)=\ell_A(A/I^2+Q)+\sum_{m=1}^{r_I-1}\ell_A(I^{m+1}/QI^m+I^{m+2})-\sum_{m=-1}^{r_I-1}\ell_{T_{\p}}(K^{(m)}_{\p})$ where ${\p}=\m T $.
\end{itemize}
\end{prop}

We need the following proposition in Section 4, where $B=T/\m T \cong (A/\m)[X_1,X_2,\cdots,X_d]$ is a polynomial ring with $d$ indeterminates over the field $A/\m$.

\begin{prop}\label{functionS}{$($c.f. \cite[Proposition 4.1]{O3}$)$}
Suppose that $I$ is stretched. 
Then the following assertions hold true:
\begin{itemize}
\item[$(1)$] We have $\overline{D^{(m)}} \cong B$ as graded $F$-modules for all $1 \leq m \leq r_I-1$, and hence we have exact sequences
$$ 0 \to K^{(m)}(-m) \to B(-m) \overset{\varphi_{m}}{\to} \overline{S^{(m)}} \to \overline{S^{(m+1)}} \to 0 $$
for $1 \leq m \leq r_I-2$, and
$$ 0 \to K^{(r_I-1)}(-r_I+1) \to B(-r_I+1) \overset{\varphi_{r_I-1}}{\to} \overline{S^{(r_I-1)}} \to 0$$
of graded $F$-modules.

\item[$(2)$] We have 
\begin{eqnarray*}
\ell_A(I^n/I^{n+1}) &=& \{\e_0(I)+r_I-k_I\}\binom{n+d-1}{d-1}\\
&-&\{\e_0(I)-\ell_A(A/I)+\binom{r_I}{2}+r_I-k_I\}\binom{n+d-2}{d-2}\\
&+& \sum_{k=2}^{r_I}(-1)^k\binom{r_I+1}{k+1}\binom{n+d-k-1}{d-k-1}- \sum_{m=-1}^{r_I-1} \ell_A(K^{(m)}_{n-m-1})
\end{eqnarray*}
for all $n \geq \max\{0, r_I-d+1\}$.
\item[$(3)$] $\displaystyle \sum_{m=-1}^{r_I-1}\ell_{T_{\p}}(K^{(m)}_{\p})=r_I-k_I=|\Lambda|$ where ${\p}=\m T $.
\end{itemize}
\end{prop}

\begin{proof}
Since $\ell_A(I^{m+1}/QI^m+I^{m+2})=1$ we have $\overline{D^{(m)}} \cong B$ as graded $F$-modules for all $1 \leq m \leq r_I-1$.
Thus other conditions follow by Proposition \ref{function} $($\cite[Proposition 4.1]{O3} also$)$.
\end{proof}

Let us now give one remark.

\begin{rem}\label{-1}
Let $n \geq 1$ be an integer.
Then $Q^{n} \cap I^{n+1}=Q^nI$ holds true if and only if $K^{(-1)}_n=(0)$.
\end{rem}

The following results play a key role for a proof of Theorem \ref{main} in Section 4.

\begin{prop}\label{ker}{$($\cite[Proposition 3.9]{O3}$)$}
Let $m \geq 0$ and $n \geq 1$ be integers. 
Then $Q^{n}I^{m+1} \cap I^{n+m+2} \subseteq Q^{n+1}I^m+Q^{n}I^{m+2}$ holds true if $K^{(m)}_n=(0)$.
\end{prop}

\begin{cor}\label{beta}{$($\cite[Corollary 3.10]{O3}$)$}
Let $n \geq \ell \geq 1$ be integers.
Assume that $K^{(m)}_{n-m-1}=(0)$ for all $-1 \leq m \leq n-\ell-1$.
Then we have $Q^{\ell}I^{n-\ell} \cap I^{n+1}=Q^{\ell}I^{n-\ell+1}$.
We especially have $QI^{n-1} \cap I^{n+1}=QI^n$ if $K^{(m)}_{n-m-1}=(0)$ for all $-1 \leq m \leq n-2$.
\end{cor}

\section{The first Hilbert coefficient of stretched ideals}

The purpose of this section is to explore the first Hilbert coefficient $\e_1(I)$ of stretched $\m$-primary ideals $I$ in terms of $k_I$.
Thanks to \cite[Theorem 4.7]{HM} we can get the following lower bound of it as follows.

\begin{prop}\label{CM}
Suppose that $I$ is stretched.
Then the inequality
$$ \e_1(I) \geq \e_0(I)-\ell_A(A/I)+\binom{k_I}{2} $$
holds true and the following conditions are equivalent to each other:
\begin{itemize}
\item[$(1)$] $ \e_1(I) = \e_0(I)-\ell_A(A/I)+\binom{k_I}{2}$,
\item[$(2)$] $r_I=k_I$, and
\item[$(3)$] $G$ is Cohen-Macaulay.
\end{itemize}
When this is the case, the following assertions are also satisfied:
\begin{itemize}
\item[$(i)$] $\ell_A(I^{n+1}/QI^n)=k_I-n$ for all $1 \leq n \leq k_I$,
\item[$(ii)$] $\displaystyle \e_k(I)=\binom{k_I+1}{k+1}$ for $2 \leq k \leq d$, and
\item[$(iii)$] the Hilbert series $ HS_I(z)$ of $I$ is given by {\small { $$ HS_I(z)=\frac{\ell_A(A/I)+\{\e_0(I)-\ell_A(A/I)-k_I+1\}z+z^2+z^3+\cdots+z^{k_I}}{(1-z)^d}.$$}} 
\end{itemize}
\end{prop}

\begin{proof}
We have $\ell_A(I^{n+1}+Q/Q)=k_I-n$ for all $1 \leq n \leq k_I$ by Lemma \ref{length} $(1)$.
Then, because $\e_1(I) \geq \e_0(I)-\ell_A(A/I)+\sum_{n \geq 1}\ell_A(I^{n+1}+Q/Q)$ by \cite[Theorem 4.7]{HM}, we get the required inequality.

The equivalence $(1) \Leftrightarrow (3)$ is satisfied by \cite[Theorem 4.7]{HM}, and other conditions $(i)$, $(ii)$, and $(iii)$ hold true by \cite[Cororally 4.2]{O3} $($see also \cite[Corollary 2.4]{S1}, \cite[Theorem 2.6]{RV3}$)$. 
\end{proof}

We can immediately get the following result of stretched $\m$-primary ideals $I$ with $k_I=3$ and the first Hilbert coefficient $\e_1(I)$ attains to the lower bound, that is the equality $\e_1(I)=\e_0(I)-\ell_A(A/I)+3$ holds true, as a corollary of Proposition \ref{CM}.

\begin{cor}\label{+3}
Suppose that $I$ is stretched and assume that the equality $\e_0(I)=\ell_A(I/I^2)-(d-1)\ell_A(A/I)+2$ $($i.e. $k_I=3)$ holds true.
Then the following three conditions are equivalent:
\begin{itemize}
\item[$(1)$] $\e_1(I)=\e_0(I)-\ell_A(A/I)+3$,
\item[$(2)$] $I^4=QI^3$, and
\item[$(3)$] $G$ is Cohen-Macaulay.
\end{itemize}
When this is the case, the following assertions also follow:
\begin{itemize}
\item[$(i)$] $\ell_A(I^3/QI^2)=1$, $r_I=3$, and $\Lambda = \emptyset$,
\item[$(ii)$] $\e_2(I)=4$ if $d \geq 3$, $\e_3(I)=1$ if $d \geq 3$, and $\e_k(I)=0$ for $4 \leq k \leq d$, and
\item[$(iii)$] the Hilbert series $HS_I(z)$ of $I$ is given by {\small { $$ HS_I(z)=\frac{\ell_A(A/I)+\{\e_0(I)-\ell_A(A/I)-2\}z+z^2+z^3}{(1-z)^d}.$$}} 
\end{itemize}
\end{cor}

The following example satisfies the conditions of Corollary \ref{+3}.

\begin{ex}
Let $A=k[[u^{7},u^{15},u^{18},u^{26},u^{27}]]$ and $Q=(u^7)$.
Then $Q$ is a reduction of the maximal ideal $\m$ of $A$, $\mu(\m)=5$, $\m^2=Q\m+(u^{30})$, and $k_{\m}=r_{\m}=3$.
Therefore, $\ell_A(A/\m^{n+1})=7(n+1)-9$ for all $n \geq 2$ so that $\e_1(\m)=\e_0(\m)+2$ and $\rmG(\m)$ is Cohen-Macaulay.
\end{ex}

It is natural to ask what happens for stretched $\m$-primary ideals $I$ with $k_I=3$ and $\e_1(I)$ attains almost minimal value, that is the case where the equality $\e_1(I)=\e_0(I)-\ell_A(A/I)+4$ holds true.

The following result plays a key role for our proof of the main theorem of this paper.

\begin{thm}\label{main}
Suppose that $I$ is stretched and assume that the equality $\e_0(I)=\ell_A(I/I^2)-(d-1)\ell_A(A/I)+2$ $($i.e. $k_I=3)$ is satisfied.
Then the following two conditions are equivalent:
\begin{itemize}
\item[$(1)$] $\e_1(I)=\e_0(I)-\ell_A(A/I)+4$,
\item[$(2)$] $\ell_A(I^3/QI^2)=\ell_A(I^4/QI^3)=1$ and $I^5=QI^4$.
\end{itemize}

When this is the case, the following assertions hold true:

\begin{itemize}
\item[$(i)$] $\Lambda=\{3\}$,
\item[$(ii)$] $\e_2(I)=7$ if $d \geq 2$, $\e_3(I)=4$ if $d \geq 3$, $\e_4(I)=1$ if $d \geq 4$, and $\e_k(I)=0$ for all $5 \leq k \leq d$,
\item[$(iii)$] the Hilbert series $HS_I(z)$ of $I$ is given by {\small { $$ HS_I(z)=\frac{\ell_A(A/I)+\{\e_0(I)-\ell_A(A/I)-2\}z+z^2+z^4}{(1-z)^d}, \ \  \mbox{and}$$}} 
\item[$(iv)$] $\depth G=d-1$.
\end{itemize}
\end{thm}

\begin{proof}
$(1) \Rightarrow (2)$, $(i)$, $(ii)$, $(iii)$, and $(iv)$:
We have $\sum_{k=-1}^{r_I-1}\ell_{T_{\p}}(K^{(k)}_{\p})=|\Lambda|=r_I-3$ by Proposition \ref{functionS} $(3)$.
The equality 
$$\ell_{T_{\p}}(\overline{S^{(2)}_{\p}}) = \sum_{k=2}^{r_I-1} \ell_{T_{\p}}(B_{\p})-\sum_{k=2}^{r_I-1}\ell_{T_{\p}}(K^{(k)}_{\p})$$ 
holds true by the inductive steps, because the equalities
$$\ell_{T_{\p}}(\overline{S_{\p}^{(m)}})=\ell_{T_{\p}}(B_{\p})-\ell_{T_{\p}}(K^{(m)}_{\p})+\ell_{T_{\p}}(\overline{S^{(m+1)}_{\p}})$$
for $2 \leq m \leq r_I-2$, and 
$$\ell_{T_{\p}}(\overline{S^{(r_I-1)}_{\p}})=\ell_{T_{\p}}(B_{\p})-\ell_{T_{\p}}(K^{(r_I-1)}_{\p})$$
hold true by the exact sequences in Proposition \ref{functionS} $(1)$.
Then, since $\ell_{T_\p}(S^{(2)}_{\p})=\e_1(I)-\e_0(I)+\ell_A(A/I)-\ell_A(I^2/QI)=2$ by Proposition \ref{S2} $(2)$ and $\overline{S^{(2)}_{\p}}$ is a homomorphic image of $S^{(2)}_{\p}$, we have
\begin{eqnarray*}
2 = \ell_{T_{\p}}(S^{(2)}_{\p}) \geq \ell_{T_{\p}}(\overline{S^{(2)}_{\p}}) &=& \sum_{k=2}^{r_I-1} \ell_{T_{\p}}(B_{\p})-\sum_{k=2}^{r_I-1}\ell_{T_{\p}}(K^{(k)}_{\p})\\
&=& r_I-2-\{ r_I-3-\sum_{k=-1}^{1}\ell_{T_{\p}}(K^{(k)}_{\p}) \} \\
& = & 1+ \sum_{k=-1}^{1}\ell_{T_{\p}}(K^{(k)}_{\p}),
\end{eqnarray*}
so that we get $\sum_{k=-1}^{1}\ell_{T_{\p}}(K^{(k)}_{\p}) \leq 1$.
Suppose that $K^{(k)}=(0)$ for all $-1 \leq k \leq 1$, then we have $[QI^{n-1} \cap I^{n+1}]/QI^n=(0)$ for all $n=2, 3$ by Corollary \ref{beta} so that $\Lambda \subseteq \{4,5,\ldots,r_I-1\}$.
However, it is impossible because $|\Lambda|=r_I-3$.
Thus, we get $ \sum_{k=-1}^{1}\ell_{T_{\p}}(K^{(k)}_{\p}) = 1 $ so that $\ell_{T_{\p}}(\overline{S^{(2)}_{\p}})=2$.
Then, because $\displaystyle \ell_{T_{\p}}(IS^{(2)}_{\p})=\ell_{T_{\p}}(S_{\p}^{(2)})-\ell_{T_{\p}}(\overline{S^{(2)}_{\p}})=0$ and $\Ass_TIS^{(2)} \subseteq \Ass_TS^{(2)} =\{\p\}$ by Proposition \ref{S2} $(1)$, we get $IS^{(2)}=(0)$.
Hence $I^{n+1} \subseteq Q^{n-2}I^2 \subseteq QI^{n-1}$ holds true so that $[QI^{n-1} \cap I^{n+1}]/QI^n=I^{n+1}/QI^n \neq (0)$ for all $3 \leq n \leq r_I-1$.
Therefore $\Lambda=\{3,4, \cdots, r_I-1\}$ because $|\Lambda|=r_I-3$.
Then since $2 \notin \Lambda$ we also get $\ell_A(I^3/QI^2)=\ell_A(I^2/QI)-1=1$ by Lemma \ref{length} $(2)$.
Thus, we get $\depth G = d-1$ by \cite[Theorem 3.2]{R} $($or \cite[Proposition 3.1]{RV3} because $I^4 \subseteq QI^2$ is satisfied$)$, and hence the equality 
$$r_I=\binom{r_I}{2}-\sum_{s \in \Lambda}s+|\Lambda|=\e_1(I)-\e_0(I)+\ell_A(A/I)=4 $$
holds true by Corollary \ref{HM} $($c.f. \cite[Theorem 3.1]{Huk} and \cite[Theorem 4.7]{HM}$)$ and $\Lambda=\{3\}$ follows.
Then, thanks to \cite[Theorem 1.1]{O3} $($\cite[Theorem 3.2]{R} and \cite[Corollary 4.3]{RV4} also$)$, the conditions $(ii)$, $(iii)$, and $(iv)$ are also satisfied.

\noindent
$(2) \Rightarrow (1)$:
We have $\Lambda = \{3\}$ by our assumption and Lemma \ref{length} $(2)$.
Thus, we can get the required equality of the first Hilbert coefficient $\e_1(I)$ by Corollary \ref{HM}, because $\e_1(I)> \e_0(I)-\ell_A(A/I)+3$ by Corollary \ref{+3}.
\end{proof}

The following example satisfies the condition of Theorem \ref{main}.

\begin{ex}
Let $A=k[[u^{8},u^{17},u^{21},u^{30},u^{39},u^{52}]]$ and $Q=(u^8)$.
Then $Q$ is a reduction of the maximal ideal $\m$ of $A$, $\mu(\m)=6$, $\m^2=Q\m+(u^{34})$, $\tau(A)=3$, $k_{\m}=3$, and $r_{\m}=4$.
Therefore, $\ell_A(A/\m^{n+1})=8(n+1)-11$ for all $n \geq 3$ so that $\e_1(\m)=\e_0(\m)+3$ and $\rmG(\m)$ is not Cohen-Macaulay.
\end{ex}

\section{Stretched ideals with small first Hilbert coefficient}

In this section let us study the structure of stretched $\m$-primary ideals with small first Hilbert coefficient.

\begin{rem}
Suppose that $I$ is stretched.
Then we notice that the inequality $\e_1(I) \geq \e_0(I)-\ell_A(A/I)+1$ holds true.
\end{rem}

Let us begin with a corollary of Proposition \ref{CM} as follows.
Corollary \ref{rank1} corresponds to the theorem of Sally \cite{S4} which characterized an $\m$-primary ideal $I$ with the equalities $\ell_A(I^2/QI)=1$ and $I^3=QI^2$.

\begin{cor}[c.f. \cite{S4}]\label{rank1}
Suppose that $I$ is stretched then the following conditions are equivalent:
\begin{itemize}
\item[$(1)$] $\e_1(I)=\e_0(I)-\ell_A(A/I)+1$ and
\item[$(2)$] $I^3=QI^2$.
\end{itemize}
When this is the case the following assertions also hold true:
\begin{itemize}
\item[$(i)$] $\ell_A(I^2/QI)=1$ $($i.e. $k_I=r_I=2)$ and $\Lambda =\emptyset$,
\item[$(ii)$] $\e_2(I)=1$ if $d \geq 2$, and $\e_k(I)=0$ for $3 \leq k \leq d$,
\item[$(iii)$] the Hilbert series $ HS_I(z)$ of $I$ is given by {\small { $$ HS_I(z)=\frac{\ell_A(A/I)+\{\e_0(I)-\ell_A(A/I)-1\}z+z^2}{(1-z)^d}, \ \mbox{and}$$}}
\item[$(iv)$] $G$ is Cohen-Macaulay. 
\end{itemize}
\end{cor}

\begin{proof}
$(1) \Rightarrow (2)$, $(i)$, $(ii)$, $(iii)$, and $(iv)$:
Because $1=\e_1(I)-\e_0(I)+\ell_A(A/I) \geq \binom{k_I}{2} \geq 1$ we get $k_I=2$ and required conditions by Proposition \ref{CM}.

\noindent
$(2) \Rightarrow (1)$:
Since $r_I=k_I=2$ by our assumption, the required equality $\e_1(I)=\e_0(I)-\ell_A(A/I)+1$ holds true by Proposition \ref{CM}.
\end{proof}

We have the following corollary in the case where the equality $\e_1(I)=\e_0(I)-\ell_A(A/I)+2$ holds true.

\begin{cor}\label{rank2}
Suppose that $I$ is stretched then the following conditions are equivalent:
\begin{itemize}
\item[$(1)$] $\e_1(I)=\e_0(I)-\ell_A(A/I)+2$ and
\item[$(2)$] $\ell_A(I^2/QI)=\ell_A(I^3/QI^2)=1$ and $I^4=QI^3$.
\end{itemize}
When this is the case, the following assertions also hold true:
\begin{itemize}
\item[$(i)$] $\Lambda=\{2\}$,
\item[$(ii)$] $\e_2(I)=3$ if $d \geq 2$, $\e_3(I)=1$ if $d \geq 3$, and $\e_k(I)=0$ for $4 \leq k \leq d$,
\item[$(iii)$] the Hilbert series $ HS_I(z)$ of $I$ is given by {\small { $$ HS_I(z)=\frac{\ell_A(A/I)+\{\e_0(I)-\ell_A(A/I)-1\}z+z^3}{(1-z)^d}, \ \mbox{and}$$}}
\item[$(iv)$] $\depth G=d-1$. 
\end{itemize}
\end{cor}

\begin{proof}
$(1) \Rightarrow (2)$, $(i)$, $(ii)$, $(iii)$ and $(iv)$:
Since $2=\e_1(I)-\e_0(I)+\ell_A(A/I) \geq \binom{k_I}{2} > 1$ and by Corollary \ref{rank1}, we have $k_I=2$ $($i.e. $\ell_A(I^2/QI)=1)$ but $I^3 \neq QI^2$ by Corollary \ref{rank1}.
Hence $\depth G \geq d-1$ by \cite[Theorem 2.1]{RV1}, \cite[Theorem 3.1]{W}, and we have $\ell_A(I^3/QI^2)=1$ by Lemma \ref{length}.
Then because $\Lambda \subseteq \{2,3,\cdots,r_I-1\}$ we get
\begin{eqnarray*}
2=\e_1(I)-\e_0(I)+\ell_A(A/I)=\binom{r_I}{2}-\sum_{s \in \Lambda}s +r_I-2 \geq  r_I-1
\end{eqnarray*}
by Corollary \ref{HM}, so that $r_I=3$.

\noindent
$(2) \Rightarrow (1)$:
We notice that, $\e_1(I)>\e_0(I)-\ell_A(A/I)+1$ by Corollary \ref{rank1}.
Therefore $\e_1(I) = \e_0(I)-\ell_A(A/I)+2$ holds true by Corollary \ref{HM}, because $\Lambda=\{2\}$, $k_I=2$, and $r_I=3$.
\end{proof}

We furthermore have the following characterization of stretched ideals with the equality $\e_1(I)=\e_0(I)-\ell_A(A/I)+3$.

\begin{prop}\label{rank3}
Suppose that $I$ is stretched then the following conditions are equivalent:
\begin{itemize}
\item[$(1)$] $\e_1(I)=\e_0(I)-\ell_A(A/I)+3$.
\item[$(2)$] Either of the following conditions holds true.
\begin{itemize}
\item[(I)] $\ell_A(I^2/QI)=2$ and $I^4=QI^3$, or
\item[(II)] $\ell_A(I^{n+1}/QI^n)=1$ for all $1 \leq n \leq 3$ and $I^5=QI^4$.
\end{itemize}
\end{itemize}
When this is the case $G$ is almost Cohen-Macaulay $($i.e. $\depth G \geq d-1)$, and the following assertions also hold true:
\begin{itemize}
\item[$(i)$] Suppose the condition $(I)$ is satisfied. Then, we have $\ell_A(I^3/QI^2)=1$ and $\Lambda=\emptyset$.
Moreover we have $\e_2(I)=4$ if $d \geq 2$, $\e_3(I)=1$ if $d \geq 3$, $\e_k(I)=0$ for $4 \leq k \leq d$, the Hilbert series $ HS_I(z)$ of $I$ is given by {\small { $$ HS_I(z)=\frac{\ell_A(A/I)+\{\e_0(I)-\ell_A(A/I)-2\}z+z^2+z^3}{(1-z)^d}, \ \mbox{and}$$}}
$G$ is Cohen-Macaulay. 
\item[$(ii)$] Suppose the condition $(II)$ is satisfied. Then, we have $\Lambda =\{2,3\}$.
Moreover we have $\e_2(I)=6$ if $d \geq 2$, $\e_3(I)=4$ if $d \geq 3$, $\e_4(I)=1$ if $d \geq 4$, $\e_k(I)=0$ for $5 \leq k \leq d$, the Hilbert series $ HS_I(z)$ of $I$ is given by {\small { $$ HS_I(z)=\frac{\ell_A(A/I)+\{\e_0(I)-\ell_A(A/I)-1\}z+z^4}{(1-z)^d}, \ \mbox{and}$$}}
$\depth G=d-1$. 
\end{itemize}
\end{prop}

\begin{proof}
$(1) \Rightarrow (2)$, $(i)$, and $(ii)$:
We have $k_I=2$ or 3 because $3=\e_1(I)-\e_0(I)+\ell_A(A/I) \geq \binom{k_I}{2} \geq 1$ by Proposition \ref{CM}.
Therefore, the assertions $(I)$ and $(i)$ are satisfied by Corollary \ref{+3} if $k_I=3$.
Suppose that $k_I=2$ then we have  $\ell_A(I^{n+1}/QI^n)=1$ for all $1 \leq n \leq 3$ by Lemma \ref{length}, Corollary \ref{rank1}, and \ref{rank2}.
Then we have $\depth G \geq d-1$ by \cite[Theorem 2.1]{RV1}, \cite[Theorem 3.1]{W}, and since $\Lambda \subseteq \{2,3,\cdots,r_I-1\}$ we get
\begin{eqnarray*}
3=\e_1(I)-\e_0(I)+\ell_A(A/I)=\binom{r_I}{2}-\sum_{s \in \Lambda}s +r_I-2 \geq  r_I-1
\end{eqnarray*}
by Corollary \ref{HM}, so that $r_I=4$, $\Lambda=\{3,4\}$, and the assertion $(ii)$ is satisfied by \cite[Theorem 2.1]{RV1}.

\noindent
$(2) \Rightarrow (1)$:
Suppose that the assertion $(I)$ is satisfied then $\e_1(I)=\e_0(I)-\ell_A(A/I)+3$ by Corollary \ref{+3}.
Suppose that the assertion $(II)$ is satisfied then we have $\Lambda=\{2,3\}$.
Hence we have $\e_1(I)=\e_0(I)-\ell_A(A/I)+3$ by Corollary \ref{HM}, because $\e_1(I)>\e_0(I)-\ell_A(A/I)+2$ by Corollary \ref{rank1} and \ref{rank2}.
\end{proof}

The main result of this paper is stated as follows.

\begin{thm}\label{rank4}
Suppose that $I$ is stretched then the following conditions are equivalent:
\begin{itemize}
\item[$(1)$] $\e_1(I)=\e_0(I)-\ell_A(A/I)+4$.
\item[$(2)$] Either of the following conditions holds true.
\begin{itemize}
\item[(I)] $\ell_A(I^2/QI)=2$, $\ell_A(I^3/QI^2)=\ell_A(I^4/QI^3)=1$, and $I^5=QI^4$, or
\item[(II)]  $\ell_A(I^{n+1}/QI^n)=1$ for all $1 \leq n \leq 4$ and $I^6=QI^5$.
\end{itemize}
\end{itemize}
When this is the case $\depth G=d-1$, and the following assertions also hold true:
\begin{itemize}
\item[$(i)$] Suppose the condition $(I)$ is satisfied. Then $\Lambda=\{3\}$.
Moreover, we have $\e_2(I)=4$ if $d \geq 2$, $\e_3(I)=1$ if $d \geq 3$, $\e_k(I)=0$ for $4 \leq k \leq d$, and the Hilbert series $ HS_I(z)$ of $I$ is given by {\small { $$ HS_I(z)=\frac{\ell_A(A/I)+\{\e_0(I)-\ell_A(A/I)-2\}z+z^2+z^4}{(1-z)^d}.$$}}
\item[$(ii)$] Suppose the condition $(II)$ is satisfied. Then $\Lambda=\{2,3,4\}$.
Moreover, we have $\e_2(I)=10$ if $d \geq 2$, $\e_3(I)=5$ if $d \geq 3$, $\e_4=5$ if $d \geq 4$, $\e_5(I)=1$ if $d \geq 5$, $\e_k(I)=0$ for $6 \leq k \leq d$, and the Hilbert series $ HS_I(z)$ of $I$ is given by {\small { $$ HS_I(z)=\frac{\ell_A(A/I)+\{\e_0(I)-\ell_A(A/I)-1\}z+z^5}{(1-z)^d}.$$}} 
\end{itemize}
\end{thm}

\begin{proof}
$(1) \Rightarrow (2)$, $(i)$, and $(ii)$:
We have $k_I=2$ or 3 because $4=\e_1(I)-\e_0(I)+\ell_A(A/I) \geq \binom{k_I}{2} \geq 1$ by Proposition \ref{CM}.
Therefore, the assertions $(I)$ and $(i)$ are satisfied by Theorem \ref{main} when $k_I=3$.

Suppose that $k_I=2$ then we have  $\ell_A(I^{n+1}/QI^n)=1$ for all $1 \leq n \leq 4$ by Lemma \ref{length}, Corollary \ref{rank1}, \ref{rank2}, and Proposition \ref{rank3}.
Then we have $\depth G \geq d-1$ by \cite[Theorem 2.1]{RV1} and \cite[Theorem 3.1]{W}.
Since $\Lambda \subseteq \{2,3,\cdots,r_I-1\}$ we get
\begin{eqnarray*}
4=\e_1(I)-\e_0(I)+\ell_A(A/I)=\binom{r_I}{2}-\sum_{s \in \Lambda}s +r_I-2 \geq  r_I-1
\end{eqnarray*}
by Corollary \ref{HM}, so that $r_I=5$, $\Lambda=\{3,4, 5\}$, and the assertion $(ii)$ is satisfied by \cite[Theorem 2.1]{RV1}.

\noindent
$(2) \Rightarrow (1)$:
Suppose that the assertion $(I)$ is satisfied then $\e_1(I)=\e_0(I)-\ell_A(A/I)+4$ by Theorem \ref{main}.
Suppose that the assertion $(II)$ is satisfied then we have $\Lambda=\{2,3, 4\}$.
Hence we have $\e_1(I)=\e_0(I)-\ell_A(A/I)+4$ by Corollary \ref{HM}, because $\e_1(I)>\e_0(I)-\ell_A(A/I)+3$ by Corollary \ref{rank1}, \ref{rank2}, and Proposition \ref{rank3}.
\end{proof}

We get the following corollary by the results of this section. 

\begin{cor}\label{cor4}
Suppose that $I$ is stretched and assume that $\e_1(I) \leq \e_0(I)-\ell_A(A/I)+4$, then $G$ is almost Cohen-Macaulay $($i.e. $\depth G \geq d-1)$.
\end{cor}

\section{Example}

In this section we introduce a few examples of one-dimensional stretched Cohen-Macaulay local rings $(A, \m)$.
See \cite[Section 6]{O3} for the detailed proofs.

\vspace{2mm}

Let $b$, $e$ and $\ell$ are integers with $b \geq 2$ and $2 \leq \ell \leq e-1$.
We set $b_1=b$ and $b_n$ $(\ell+1 \leq n \leq e-1)$ be integers with $\lceil \frac{b}{2} \rceil n +1 \leq b_n$ for all $\ell+1 \leq n \leq e-1$, $b_{\ell+1} \leq b\ell+b-1$, and $b_{n+1} \leq b_{n}+\lceil \frac{b}{2} \rceil$ for all $\ell+1 \leq n \leq e-2$ where $\lceil q \rceil =\min\{n \in \Z \ | \ n \geq q \}$ for a quotient number $q \in \mathbb{Q}$.
We notice here that the inequalities $\lceil \frac{b}{2} \rceil n+1 \leq b_n \leq (b-1)n+\ell$ hold true for $\ell+1 \leq n \leq e-1$.
We set $r=\max\{ n < e \ | \ b_n > bn-n+1 \} \cup \{\ell\}$.

Let $H=\langle e, b_1e+1, \{b_ne+n\}_{\ell+1 \leq n \leq e-1} \rangle$ be the numerical semi-group generated by $e$, $b_1e+1$, and $b_{n}e+n$ for $\ell+1 \leq n \leq e-1$.
We set $C=k[H]=k[ u^e, u^{be+1},\{u^{  b_{n}e+n }\}_{\ell+1 \leq n \leq e-1} ] \subseteq k[u]$ be a numerical semi-group ring of $H$, where $k[u]$ denotes the polynomial ring with one indeterminate $u$ over a field $k$.
We set $\q=u^e C$ and $\n$ is the graded maximal ideal of $C$.

Let $A=C_{\n}$ be a localization of $C$ at $\n$ and $\m=\n A$ the maximal ideal of $A$.
We set $Q=\q A$ and then $Q$ is a parameter ideal in $A$ which forms reduction of $\m$.
We then have the following proposition.

\begin{prop}\label{exthm}{$($\cite[Theorem 6.6]{O3}$)$}
The following assertions hold true.
\begin{itemize}
\item[$(1)$] $A$ is a stretched Cohen-Macaulay local ring with $\dim A=1$ and $\tau(A)=\mu_A(\m)-1=e-\ell$.
\item[$(2)$] $k_{\m}=\ell$ and $r_{\m}=r$.
\item[$(3)$] $\Lambda_{\m}=\{bn-b_n+1 \ | \ \ell+1 \leq n \leq r \}.$
\item[$(4)$] $\ell_A(A/\m^{n+1})=e(n+1)-(e-1+\binom{\ell}{2}+\sum_{n=\ell+1}^{r}(b_n-bn+n-1))$ for all $n \geq r-1$.
\item[$(5)$] The following two conditions are equivalent to each other:
\begin{itemize}
\item[$(i)$] ${\rmG}(\m)$ is Cohen-Macaulay and
\item[$(ii)$] $b_{\ell+1} \leq b\ell+b-\ell$ that is the equality $r=\ell$ holds true.
\end{itemize}
\end{itemize}
\end{prop}

The following Example \ref{ex1} and \ref{ex2} corresponds to Corollary \ref{rank1} and \ref{rank2} respectively.

\begin{ex}\label{ex1}
Let $\ell=2$ and $e \geq 3$ in the above setting, and we set $b_{3}=3b-2$.
Then we have the following.
\begin{itemize}
\item[$(1)$] $k_{\m}=r_{\m}=2$.
\item[$(2)$] $\Lambda_{\m}=\emptyset$.
\item[$(3)$] $\ell_A(A/\m^{n+1})=e(n+1)-e$ for all $n \geq 1$. Therefore $\e_1(\m)=\e_0(\m)$.
\item[$(4)$] $\rmG(\m)$ is Cohen-Macaulay.
\end{itemize}
\end{ex}

\begin{ex}\label{ex2}
Let $\ell=2$ and $e \geq 4$ in the above setting, and we set $b_{3}=3b-1$ and $b_4=4b-3$.
Then we have the following.
\begin{itemize}
\item[$(1)$] $k_{\m}=2$ and $r_{\m}=3$.
\item[$(2)$] $\Lambda_{\m}=\{2\}$.
\item[$(3)$] $\ell_A(A/\m^{n+1})=e(n+1)-(e+1)$ for all $n \geq 2$. Therefore $\e_1(\m)=\e_0(\m)+1$.
\item[$(4)$] $\rmG(\m)$ is not Cohen-Macaulay.
\end{itemize}
\end{ex}

The following Example \ref{ex3-1} and \ref{ex3-2} satisfies the condition of Proposition \ref{rank3} $(I)$ and $(II)$ respectively.

\begin{ex}\label{ex3-1}
Let $\ell=3$ and $e \geq 5$ in the above setting, and we set $b_{4}=4b-3$.
Then we have the following.
\begin{itemize}
\item[$(1)$] $k_{\m}=r_{\m}=3$.
\item[$(2)$] $\Lambda_{\m}=\emptyset$.
\item[$(3)$] $\ell_A(A/\m^{n+1})=e(n+1)-(e+2)$ for all $n \geq 2$. Therefore $\e_1(\m)=\e_0(\m)+2$.
\item[$(4)$] $\rmG(\m)$ is Cohen-Macaulay.
\end{itemize}
\end{ex}

\begin{ex}\label{ex3-2}
Let $\ell=2$ and $e \geq 5$ in the above setting, and we set $b_{3}=3b-1$, $b_4=4b-2$, and $b_5=5b-4$.
Then we have the following.
\begin{itemize}
\item[$(1)$] $k_{\m}=2$ and $r_{\m}=4$.
\item[$(2)$] $\Lambda_{\m}=\{2,3\}$.
\item[$(3)$] $\ell_A(A/\m^{n+1})=e(n+1)-(e+2)$ for all $n \geq 3$. Therefore $\e_1(\m)=\e_0(\m)+2$.
\item[$(4)$] $\rmG(\m)$ is not Cohen-Macaulay.
\end{itemize}
\end{ex}

We can construct $1$-dimensional stretched Cohen-Macaulay local ring whose maximal ideal satisfying the condition of Theorem \ref{rank4} as the following Example \ref{ex4-1} and \ref{ex4-2}.

\begin{ex}\label{ex4-1}
Let $\ell=3$ and $e \geq 5$ in the above setting, and we set $b_{4}=4b-2$ and $b_5=5b-4$.
Then we have the following.
\begin{itemize}
\item[$(1)$] $k_{\m}=3$ and $r_{\m}=4$.
\item[$(2)$] $\Lambda_{\m}=\{3\}$.
\item[$(3)$] $\ell_A(A/\m^{n+1})=e(n+1)-(e+3)$ for all $n \geq 3$. Therefore $\e_1(\m)=\e_0(\m)+3$.
\item[$(4)$] $\rmG(\m)$ is not Cohen-Macaulay.
\end{itemize}
\end{ex}

\begin{ex}\label{ex4-2}
Let $\ell=2$ and $e \geq 6$ in the above setting, and we set $b_{3}=3b-1$, $b_4=4b-2$, $b_5=5b-3$, and $b_6=6b-5$.
Then we have the following.
\begin{itemize}
\item[$(1)$] $k_{\m}=2$ and $r_{\m}=5$.
\item[$(2)$] $\Lambda_{\m}=\{2,3,4\}$.
\item[$(3)$] $\ell_A(A/\m^{n+1})=e(n+1)-(e+3)$ for all $n \geq 4$. Therefore $\e_1(\m)=\e_0(\m)+3$.
\item[$(4)$] $\rmG(\m)$ is not Cohen-Macaulay.
\end{itemize}
\end{ex}

In the end of this paper, let us introduce some concrete examples of stretched local rings whose maximal ideal, respectively, satisfy conditions of Corollary \ref{rank1}, \ref{rank2}, Theorem \ref{rank3}, or \ref{rank4} in Section 5.

\begin{ex}
Let $e=6$ and $b=2$.
Then the following assertions hold true.
\begin{itemize}
\item[$(1)$] Let $\ell=2$, $b_3=4$, $b_4=5$, and $b_5=6$. 
Then $H=\langle 6,13,27,34,41 \rangle$ and $r=2$.
Therefore, $r_{\m}=k_{\m}=2$, $\Lambda_{\m}=\emptyset$,
$\ell_A(A/\m^{n+1})=6(n+1)-6$ for all $n \geq 2$, and $\rmG(\m)$ is Cohen-Macaulay.
\item[$(2)$] Let $\ell=2$, $b_3=b_4=5$, and $b_5=6$.
Then $H=\langle 6,13,33, 34,41 \rangle$ and $r=3$.
Therefore, $r_{\m}=3$, $k_{\m}=2$, $\Lambda_{\m}=\{2\}$, $\ell_A(A/\m^{n+1})=6(n+1)-7$ for all $n \geq 2$, and ${\rmG}(\m)$ is not Cohen-Macaulay.
\item[$(3)$] Let $\ell=3$, $b_4=5$, and $b_5=6$.
Then $H=\langle 6,13,34, 41 \rangle$ and $r=3$.
Therefore, $r_{\m}=k_{\m}=3$, $\Lambda_{\m}=\emptyset$, $\ell_A(A/\m^{n+1})=6(n+1)-8$ for all $n \geq 2$, and $\rmG(\m)$ is Cohen-Macaulay.
\item[$(4)$] Let $\ell=2$, $b_3=5$, and $b_4=b_5=6$.
Then $H=\langle 6,13,33,40, 41 \rangle$ and $r=4$.
Therefore, $r_{\m}=4$, $k_{\m}=2$, $\Lambda_{\m}=\{2,3\}$, $\ell_A(A/\m^{n+1})=6(n+1)-8$ for all $n \geq 3$, and $\rmG(\m)$ is not Cohen-Macaulay.
\item[$(5)$] Let $\ell=3$ and $b_4=b_5=6$.
Then $H=\langle 6,13,40,41 \rangle$ and $r=4$.
Therefore, $r_{\m}=4$, $k_{\m}=3$, $\Lambda_{\m}=\{3\}$, $\ell_A(A/\m^{n+1})=6(n+1)-9$ for all $n \geq 3$, and $\rmG(\m)$ is not Cohen-Macaulay.
\item[$(6)$] Let $\ell=2$, $b_3=5$, $b_4=6$, and $b_5=b_6=7$.
Then $H=\langle 6,13,33,40, 47,48 \rangle$ and $r=5$.
Therefore, $r_{\m}=5$, $k_{\m}=2$, $\Lambda_{\m}=\{2,3,4\}$, $\ell_A(A/\m^{n+1})=6(n+1)-9$ for all $n \geq 4$, and $\rmG(\m)$ is not Cohen-Macaulay.
\end{itemize}
\end{ex}



\end{document}